\documentclass[11pt]{amsart}
\usepackage{amssymb}
\usepackage{graphics}
\usepackage{hyperref}
\usepackage[all]{xy}
\usepackage{enumerate}
\usepackage[mathscr]{eucal}

\theoremstyle{plain}
\newtheorem{theorem}{Theorem}[section]
\newtheorem{lemma}[theorem]{Lemma}
\newtheorem{proposition}[theorem]{Proposition}
\newtheorem{corollary}[theorem]{Corollary}
\numberwithin{equation}{section}

\theoremstyle{definition}

\newtheorem{definition}[theorem]{Definition}

\newtheorem{remark}[theorem]{Remark}
\newtheorem{notation}[theorem]{Notation}

\DeclareMathOperator{\Ob}{Ob}
\DeclareMathOperator{\Mor}{Mor}

\DeclareMathOperator{\Module}{-Mod}
\DeclareMathOperator{\module}{-mod}
\DeclareMathOperator{\Hom}{Hom}

\DeclareMathOperator{\Ext}{Ext}
\DeclareMathOperator{\Ker}{Ker}
\DeclareMathOperator{\op}{op}
\DeclareMathOperator{\irr}{Irr}

\newcommand{\C}{{\mathscr{C}}}
\newcommand{\D}{{\mathsf{D}}}
\newcommand{\F}{{\mathbb{F}}}
\newcommand{\FI}{{\mathrm{FI}}}
\newcommand{\im}{{\mathrm{Im}}}
\newcommand{\J}{{\mathscr{J}}}
\newcommand{\la}{{\underline{\lambda}}}
\newcommand{\PP}{{\mathrm{P}}}
\newcommand{\q}{{\mathsf{Q}}}
\newcommand{\s}{{\mathsf{S}}}
\newcommand{\VI}{{\mathrm{VI}}}
\newcommand{\Z}{{\mathbb{Z}}}

\makeatletter

\newcommand\bigCircledast{\mathop{\mathpalette\b@gCircledast\relax}}
\newcommand\b@gCircledast[2]{
\vcenter{\hbox{\m@th
\scalebox{\ifx#1\displaystyle 2.3\else1.2\fi}{$#1\circledast$}
}}}

\makeatother

\title{Coinduction functor in representation stability theory}

\author{Wee Liang Gan}
\address{Department of Mathematics, University of California, Riverside, CA 92521, USA}
\email{wlgan@math.ucr.edu}

\author{Liping Li}
\address{College of Mathematics and Computer Sciences, Hunan Normal University, Changsha, 410081, China.
}
\email{lipingli@hunnu.edu.cn}

\begin{document}

\begin{abstract}
We study the coinduction functor on the category of FI-modules and its variants. Using the coinduction functor, we give new proofs of (generalizations of) various results on homological properties of FI-modules. We also prove that any finitely generated projective VI-module  over a field of characteristic 0 is injective.
\end{abstract}


\maketitle

\tableofcontents

\section{Introduction}

\subsection{Conventions}
Let $k$ be a commutative ring. By a category, we shall always mean a small category. If $\C$ is a category, a $\C$-module over $k$ is a functor from $\C$ to the category of $k$-modules. We shall refer to $\C$-modules over $k$ simply as $\C$-modules. A morphism of $\C$-modules is a natural transformation of functors. We shall write $\C\Module$ for the category of $\C$-modules, and $\C\module$ for the category of finitely generated $\C$-modules.

If $\C$ is a category and $m,n \in\Ob(\C)$, we shall write $k\C(m,n)$ for the free $k$-module on the set $\C(m,n)$. Let $k\C = \bigoplus_{m,n\in\Ob(\C)} k\C(m,n)$ and denote by $e_n \in \C(n,n)$ the identity morphism of $n\in \Ob(\C)$. We have a natural algebra structure on $k\C$ where the product of $\alpha\in \C(r,n)$ and $\beta\in \C(m,l)$ is defined to be the composition $\alpha\beta$ if $r=l$; it is defined to be $0$ if $r\neq l$. A $k\C$-module $V$ is said to be \emph{graded} if $V=\bigoplus_{n\in \Ob(\C)} e_n V$. If $V$ is a $\C$-module, then $\bigoplus_{n\in \Ob(C)} V(n)$ has a natural structure of a graded $k\C$-module which we shall also denote by $V$. Conversely, any graded $k\C$-module $V$ defines a natural $\C$-module denoted again by $V$, with $V(n)=e_n V$. Thus, we shall not distinguish between the notion of $\C$-modules and the notion of graded $k\C$-modules.

\subsection{Main results}
Let $\Z_+$ be the set of non-negative integers. For any $n\in\Z_+$, we write $[n]$ for the set $\{1,\ldots, n\}$; in particular, $[0]=\emptyset$. Let $G$ be a finite group.

\begin{definition}
Let $\FI_G$ be the category whose set of objects is $\Z_+$, and whose morphisms from $m$ to $n$ are all pairs $(f,c)$ where $f:[m]\to [n]$ is an injective map and $c:[m]\to G$ is an arbitrary map. The composition of morphisms $(f_1,c_1)\in \FI_G(m,l)$ and $(f_2,c_2)\in \FI_G(l,n)$ is
defined by $(f_2,c_2)(f_1,c_1) = (f_3,c_3)$ where
\begin{equation*}
f_3(t) = f_2(f_1(t)), \quad c_3(t) = c_2(f_1(t))c_1(t), \quad \mbox{ for all } t\in [m].
\end{equation*}
We write $\FI$ for $\FI_G$ when $G$ is trivial.
\end{definition}

Let $\F$ be a finite field.

\begin{definition}
Let $\VI$ be the category whose set of objects is $\Z_+$, and whose morphisms from $m$ to $n$ are all injective linear maps from $\F^m$ to $\F^n$. The composition of morphisms is defined to be the composition of maps.
\end{definition}

Suppose $\C$ is the category $\FI_G$ or $\VI$. There is a natural monoidal structure $\odot$ on $\C$ defined on objects by $m \odot n = m+n$. Define a functor $\iota : \C \to \C$ by
\begin{equation} \label{iota}
\iota(n) = 1\odot n, \quad \iota(\alpha) = e_1\odot \alpha, \quad \mbox{ for each } n\in \Ob(\C),\; \alpha\in \Mor(\C).
\end{equation}
The functor $\iota$ is faithful and gives rise to a restriction functor:
\begin{equation*}
\s  : \C\Module \longrightarrow \C\Module, \quad V \mapsto V\circ \iota.
\end{equation*}
From the usual tensor-hom adjunction, one can define a right adjoint functor $\q$ to $\s$ called the coinduction functor. The main results of this paper are the following.

\begin{theorem} \label{main result on q 1}
Suppose that $k$ is a commutative ring, and $\C$ is the category $\FI_G$. Let $m \in \Ob(\C)$. Then $\q(k \C e_m)$ is isomorphic to $k\C e_m\oplus k\C e_{m+1}$.
\end{theorem}

The proof of Theorem \ref{main result on q 1} will be given in Section \ref{FI_G section}.

Denote by $q$ the number of elements of  $\F$.

\begin{theorem} \label{main result on q 2}
Suppose that $k$ is a commutative ring, and $\C$ is the category $\VI$. Let $m\in \Ob(\C)$. If $q$ is a unit of $k$, then $\q(k\C e_m)$ contains a direct summand isomorphic to $k\C e_{m+1}$.
\end{theorem}

To prove Theorem \ref{main result on q 2}, we construct a non-obvious surjective homomorphism from $\q(k\C e_m)$ to $k\C e_{m+1}$. The details will be given in Section \ref{VI section}.

\subsection{Applications}
Let us indicate quickly the use of the coinduction functor in this paper.

Suppose $\C$ is the category $\FI_G$ or $\VI$. It is known that $\s(V)$ is projective if $V$ is a finitely generated projective $\C$-module. Hence, by the Eckman-Shapiro lemma, one has $\Ext^1_\C(\s(V),W) = \Ext^1_\C(V, \q(W))$ for any $\C$-modules $V$ and $W$. It follows that $\q(W)$ is injective if $W$ is injective.

Suppose now that $k$ is a field of characteristic 0. In this case, $\C$ is locally Noetherian. By reducing to finite categories, one can show that $k\C e_0$ is an injective $\C$-module. It follows by induction, using Theorems \ref{main result on q 1} and \ref{main result on q 2}, that $k\C e_n$ is injective for all $n\in\Ob(\C)$. Hence, we deduce the following result.

\begin{theorem} \label{main-result-1}
Suppose that $k$ is a field of characteristic 0, and $\C$ is the category $\FI_G$ or $\VI$. Let $V$ be a finitely generated projective $\C$-module. Then $V$ is an injective $\C$-module.
\end{theorem}

We shall show, in Section \ref{inj}, that if $k$ is a field of positive characteristic and $\C$ is $\FI$, the projective $\C$-module $k\C e_0$ is not injective.

\begin{corollary} \label{infinite global dimension}
Suppose that $k$ is a field of characteristic 0, and $\C$ is the category $\FI_G$ or $\VI$. Let $V$ be a finitely generated $\C$-module. Then $V$ has a finite projective resolution in the category $\C\module$ if and only if $V$ is a projective $\C$-module.
\end{corollary}

If $F$ is a nonzero finitely generated torsion-free $\C$-module, then there is a smallest $a\in \Ob(\C)$ such that $F(a)\neq 0$, and it is easy to see that $\s^a(F)$ contains $k\C e_0$ as a $\C$-submodule. Since $k\C e_0$ is injective, it is a direct summand of $\s^a(F)$. By the adjunction of $\s$ and $\q$, it follows that there exists a nonzero homomorphism from $F$ to $\q^a (k\C e_0)$. Thus, when $\C$ is $\FI_G$, it follows from Theorem \ref{main result on q 1} that every nonzero finitely generated torsion-free $\C$-module has a nonzero homomorphism to a finitely generated projective $\C$-module. From this, it is not difficult to deduce the following result.

\begin{theorem} \label{main-result-2}
Suppose that $k$ is a field of characteristic 0, and $\C$ is the category $\FI_G$. Then one has the following.

(i) Any finitely generated injective $\C$-module is a direct sum of a finite dimensional injective $\C$-module and a finitely generated projective $\C$-module.

(ii) Any finitely generated $\C$-module has a finite injective resolution in the category of finitely generated $\C$-modules.
\end{theorem}

The proofs of Theorems \ref{main-result-1} and \ref{main-result-2} will be given in Section \ref{applications of coinduction}. We do not know if Theorem \ref{main-result-2} holds for the category $\VI$.

\begin{remark}
Theorems \ref{main-result-1} and \ref{main-result-2} were first proved for the category $\FI$ by Sam and Snowden in \cite{SS-GL}. Our proofs are new and independent of their results.
\end{remark}

\begin{remark}
In \cite{Kuhn}, Kuhn studied the category of $\C$-modules where $\C$ is the category of finite dimensional $\F$-vector spaces with all linear maps as morphisms. He showed that if $k$ is a field and $q$ is a unit in $k$, the category of $\C$-modules is equivalent to the product over all $n\in \Z_+$ of the categories of $\mathrm{GL}_n(\F)$-modules; see \cite[Theorem 1.1]{Kuhn}. As a consequence, if $k$ is a field of characteristic 0, the category of $\C$-modules is semisimple; see \cite[Corollary 1.3]{Kuhn}. In contrast, the categories $\FI_G\Module$ and $\VI\Module$ do not have a similar decomposition, and Corollary \ref{infinite global dimension} implies that the categories $\FI_G\module$ and $\VI\module$ have infinite global dimension when characteristic of $k$ is 0.
\end{remark}

\subsection{Representation stability} \label{rep stability}
An upshot of Theorem \ref{main-result-2} is an alternative proof of a key theorem of  Church, Ellenberg, and Farb in their theory of representation stability. We shall discuss this in the more general situation of wreath product groups.

Suppose that $k$ is a splitting field for $G$ of characteristic 0. Denote by $\irr(G)$ the set of isomorphism classes of simple $kG$-modules, and $\chi_1 \in \irr(G)$ the trivial class. For each $n\in \Z_+$, the isomorphism classes of simple $k G\wr S_n$-modules are parametrized by partition-valued functions $\la$ on $\irr(G)$ such that $|\la|=n$, where
\begin{equation*}
|\la| = \sum_{\chi\in\irr(G)} |\la(\chi)|.
\end{equation*}

Suppose $\la$ is any partition-valued function on $\irr(G)$ and $n$ is any integer $\geqslant |\la|+a$, where $a$ is the biggest part of $\la(\chi_1)$. Following \cite{SS-G-maps},  we define the partition-valued function $\la[n]$ on $\irr(G)$ by
\begin{equation*}
\la[n](\chi) = \left\{ \begin{array}{ll}
(n-|\la|, \la(\chi_1)) & \mbox{ if } \chi=\chi_1,\\
\la(\chi) & \mbox{ if } \chi\neq \chi_1.
\end{array} \right.
\end{equation*}
Let $L(\la)_n$ be a simple $k G\wr S_n$-module belonging to the isomorphism class corresponding to $\la[n]$.

Suppose $V_n$ is a sequence of $k G\wr S_n$-modules equipped with linear maps $\phi_n : V_n \to V_{n+1}$. We say that $\{V_n\}$ is a \emph{consistent} sequence if the following diagram commutes for each $g\in G \wr S_n$:
\begin{equation*}
\xymatrix{ V_n \ar[r]^{\phi_n} \ar[d]_{g} & V_{n+1} \ar[d]^{g} \\ V_n \ar[r]_{\phi_n} & V_{n+1} }
\end{equation*}
Here, $g$ acts on $V_{n+1}$ by its image under the standard inclusion $G\wr S_n \hookrightarrow G\wr S_{n+1}$.

\begin{definition}
A consistent sequence $\{V_n\}$ of $kG\wr S_n$-modules is \emph{representation stable} if there exists $N>0$ such that for each $n\geqslant N$, the following three conditions hold:
\begin{itemize}
\item[(RS1)]
{\bf Injectivity:} The map $\phi_n : V_n \longrightarrow V_{n+1}$ is injective.

\item[(RS2)]
{\bf Surjectivity:} The span of the $G\wr S_{n+1}$-orbit of $\phi_n(V_n)$ is all of $V_{n+1}$.

\item[(RS3)]
{\bf Multiplicities:}
There is a decomposition
\begin{equation*}
V_n  = \bigoplus_{\la} L(\la)_n ^{\oplus c(\la)}
\end{equation*}
where the multiplicities $0\leqslant c(\la) \leqslant \infty$ do not depend on $n$.
\end{itemize}
\end{definition}

\begin{remark}
Our terminology of representation stable follows \cite{Farb}. In \cite{CF} and \cite{CEF}, this is called \emph{uniformly} representation stable.
\end{remark}

Suppose $\C$ is $\FI_G$. Let $(\iota_n, c_n)\in \Hom_\C (n,n+1)$ be the morphism where $\iota_n : [n] \hookrightarrow [n+1]$ is the standard inclusion and $c_n:[n] \to G$ is the constant map whose image is the identity element of $G$. If $V$ is a $\C$-module, then $\{ V(n) \}$ is a consistent sequence of $kG \wr S_n$-modules where the maps $\phi_n : V(n)\to V(n+1)$ are induced by the morphisms $(\iota_n, c_n)$. The following theorem was first proved in \cite[Theorem 1.13]{CEF} when $G$ is trivial.

\begin{theorem}[Finite generation vs. representation stability] \label{rep stable thm}
Suppose that $k$ is a splitting field for $G$ of characteristic 0. Let $V$ be a $\FI_G$-module. Then $V$ is finitely generated if and only if $\{V(n)\}$ is a representation stable sequence of $k G\wr S_n$-modules with $\dim_k V(n)<\infty$ for each $n$.
\end{theorem}

It is easy to see that $V$ is finitely generated if and only if condition (RS2) holds and $\dim_k V(n)< \infty$ for each $n$; see \cite[Proposition 5.2]{GL}. Suppose that $V$ is a finitely generated $\C$-module. Then condition (RS1) is a simple consequence of the fact that $V$ is Noetherian; see \cite[Proposition 5.1]{GL}. The real task is to prove that condition (RS3) holds. But by Theorem \ref{main-result-2}, it suffices to verify condition (RS3) for finitely generated projective $\C$-modules, and this is easily accomplished by Pieri's formula. We shall give the details of the proof of Theorem \ref{rep stable thm} in Section \ref{last section}.

\subsection{Bibliographical remarks}
Church and Farb introduced the notion of representation stability for various families of groups in \cite{CF}. The connection of representation stability for sequences of $S_n$-representations to $\FI$-modules was subsequently made by Church, Ellenberg, and Farb in \cite{CEF}; their paper contains many interesting examples of representation stable sequences in algebra, geometry and topology. It was also proved in \cite[Theorem 1.3]{CEF} that $\FI$ is locally Noetherian over any field $k$ of characteristic 0; this result was later extended to an arbitrary Noetherian ring $k$ by  Church, Ellenberg, Farb, and Nagpal in \cite[Theorem A]{CEFN}.

In \cite{Wilson}, Wilson defined and studied $\FI_{\mathcal{W}}$-modules associated to the Weyl groups of type B/C and type D. In particular, she gave a proof of the analogue of Theorem \ref{rep stable thm} in these cases along the same lines as the proof in \cite{CEF}. Using \cite[Theorem A]{CEFN}, she proved that $\FI_{\mathcal{W}}$ is locally Noetherian over any Noetherian ring $k$. In the type B/C case, the category $\FI_{\mathcal{W}}$ is same as the category $\FI_G$ for the group $G$ of order 2.

From a different point of view, Snowden \cite[Theorem 2.3]{Snowden} also proved that $\FI$ is locally Noetherian over any field $k$ of characteristic 0. In fact, he proved this for any twisted commutative algebra finitely generated in order 1. Theorems \ref{main-result-1} and  \ref{main-result-2} for the category $\FI$ were proved by Sam and Snowden in \cite[Corollary 4.2.5]{SS-GL}, \cite[Theorem 4.3.1]{SS-GL}, and \cite[Theorem 4.3.4]{SS-GL}. The strategy of their proofs is to pass to the Serre quotient of the category of finitely generated $\FI$-modules by the subcategory of finite dimensional $\FI$-modules.
Using their results, Sam and Snowden deduced a formula for the character polynomial of a finitely generated $\FI$-module over a field of characteristic 0.

In \cite{GL}, we gave a simple proof that $\FI_G$ and $\VI$ are locally Noetherian over any field $k$ of characteristic 0 (which is sufficient for the present paper); a similar proof for $\FI$ and $\VI$ was also obtained by Putman (unpublished). At about the same time, using Gr\"{o}bner basis methods, Putman and Sam \cite[Theorem A]{PS} proved that $\VI$ is locally Noetherian when $\F$ is any finite ring and $k$ is any Noetherian ring, and Sam and Snowden \cite[Theorem 10.1.2]{SS-Grobner} proved that $\FI_G$ is locally Noetherian for any Noetherian ring $k$. In \cite[Theorem 1.2.4]{SS-G-maps}, Sam and Snowden proved that if $k$ is a field in which the order of $G$ is invertible, then representations of $\FI_G$ are in fact equivalent to representations of $\FI \times \mathrm{FB}^{a}$, where $\mathrm{FB}$ is the groupoid of finite sets and $a$ is the number of non-trivial irreducible representations of $G$. Using this, they proved a wreath-product version of Murnaghan's stability theorem \cite[Theorem 5.2.1]{SS-G-maps}.

It was first observed by Church, Ellenberg, Farb, and Nagpal in \cite[Proposition 2.12]{CEFN} that the
functor $\s$ for the category $\FI$ has the property that $\s(V)$ is projective whenever $V$ is a finitely generated projective $\FI$-module. A more precise version of this property plays a crucial role in their paper. In \cite[Section 5]{GL-Koszulity}, we showed that the functor $\s$ for many other categories have this property, including $\FI_G$ and $\VI$. To the best of our knowledge, no one has studied the coinduction functor $\q$, even in the case of $\FI$.

\section{Preliminaries} \label{preliminaries}

\subsection{Notations and terminology} \label{notations and terminology}
Recall that an EI category is a category in which every endomorphism is an isomorphism (see \cite{Dieck}). Throughout this paper, we shall denote by $\C$ an EI category satisfying the following conditions:

\begin{itemize}
\item
$\Ob(\C)=\Z_+$;

\item
$\C(m,n)$ is an empty set if $m>n$;

\item
$\C(m,n)$ is a nonempty finite set if $m\leqslant n$;

\item
for $m\leqslant l\leqslant n$, the composition map $\C(l,n) \times \C(m,l) \to \C(m,n)$ is surjective.
\end{itemize}

Suppose $n\in \Ob(\C)$. We shall use the following notations:
\begin{itemize}
\item $G_n$ denotes the group $\C(n,n)$;

\item $\C_n$ denotes the full subcategory of $\C$ with $\Ob(\C_n)=\{0, 1, \ldots, n\}$.
\end{itemize}

Suppose that $k$ is a commutative ring and $V$ is a $\C$-module. For $n\in \Ob(\C)$ and $v\in V(n)$, we say that $v$ has degree $n$, and write $\deg(v)$ for the degree of $v$. By a set of generators of $V$, we shall always mean a subset $S$ of $\bigcup_{n\in \Ob(\C)} V(n)$ such that the only submodule of $V$ containing $S$ is $V$ itself. We say that $V$ is generated in degrees $\leqslant n$ if $V$ has a set $S$ of generators such that the degree of each element of $S$ is at most $n$. We say that $V$ is finitely generated if $V$ has a finite set of generators.
We say that $V$ is Noetherian if every submodule of $V$ is finitely generated. We say that $\C$ is locally Noetherian if every finitely generated $\C$-module is Noetherian.

\subsection{Baer's criterion}
We omit the proof of the following lemma, which is standard  (see \cite[page 39]{Weibel}).

\begin{lemma}[Baer's criterion]
Suppose that $k$ is a commutative ring, and $V\in \C\Module$. Suppose that for all $n\in \Ob(\C)$ and for all $\C$-submodule $U$ of $k\C e_n$, every homomorphism $U \to V$ can be extended to a homomorphism $k\C e_n \to V$. Then $V$ is injective in $\C\Module$.
\end{lemma}

The following corollary is immediate.

\begin{corollary} \label{injectivity in C-mod}
Suppose that $k$ is a commutative ring and $\C$ is locally Noetherian. Let $V\in \C\module$. Then $V$ is injective in $\C\Module$ if and only if $V$ is injective in $\C\module$.
\end{corollary}

\subsection{Projective resolutions}
Suppose that $k$ is a commutative ring. For any $n\in \Ob(\C)$, the $\C$-module $k\C e_n$ is clearly projective. Suppose $V$ is a $\C$-module and $S$ is a set of generators of $V$. Then there is a surjective homomorphism $\bigoplus_{s\in S} k\C e_{\deg(s)} \to V$ whose restriction to the direct summand corresponding to $s$ is $\alpha\mapsto \alpha s$.  Hence, any $\C$-module $V$ has a projective resolution
\begin{equation*}
\cdots \to P^{-2}  \to P^{-1} \to P^{0} \to V \to 0
\end{equation*}
such that each $P^{-i}$ is a direct sum of projective $\C$-modules of the form $k\C e_n$ where $n\in \Ob(\C)$.

\subsection{Restriction to $\C_n$}
Suppose that $k$ is a commutative ring. Let $n\in \Ob(\C)$. Recall that $\C_n$ denotes the full subcategory of $\C$ with $\Ob(\C_n)=\{0, 1, \ldots, n\}$; see Subsection \ref{notations and terminology}.
Denote by $\jmath: \C_n \hookrightarrow \C$ the inclusion functor. We have the pullback functor
\begin{equation*}
 \jmath^*: \C\Module \longrightarrow \C_n\Module, \quad  V \mapsto V\circ \jmath.
\end{equation*}
We also have the pushforward functor
\begin{equation*}
\jmath_*: \C_n\Module \longrightarrow \C\Module
\end{equation*}
which regards a $\C_n$-module as a $\C$-module in the obvious way.
The pushforward functor $\jmath_*$ is a right adjoint functor to the pullback functor $\jmath^*$.

\begin{lemma} \label{eckmann-shapiro1}
Suppose that $k$ is a commutative ring. Let $n\in\Ob(\C)$ and denote by $\jmath: \C_n \hookrightarrow \C$ the inclusion functor. For any $V\in \C\Module$ and $W\in \C_n\Module$, one has
\begin{equation*}
\Ext_{\C_n}^i(\jmath^*(V), W) = \Ext_{\C}^i(V,\jmath_*(W))
\end{equation*}
for all $i\geqslant 1$.
\end{lemma}

\begin{proof}
Observe that $\jmath^*(k\C e_m)$ is a projective $\C_n$-module for all $m\in \Ob(\C)$. Thus, the required result follows from the Eckmann-Shapiro lemma.
\end{proof}

\begin{lemma} \label{reduction to C_n}
Suppose that $k$ is a commutative ring and $\C$ is locally Noetherian. Let $V, W \in \C\module$. Then there exists $N \in \Ob(\C)$ such that for all $n\geqslant N$, one has
\begin{equation*}
\Ext_{\C}^1(V, W) = \Ext_{\C_n}^1(\jmath^*(V),\jmath^*(W)),
\end{equation*}
where $\jmath: \C_n \hookrightarrow \C$ is the inclusion functor.
\end{lemma}

\begin{proof}
Since $\C$ is locally Noetherian, there exists a projective resolution
\begin{equation*}
\cdots \to P^{-2}  \to P^{-1} \to P^{0} \to V \to 0
\end{equation*}
such that each $P^{-i}$ is a finitely generated projective $\C$-module. Thus, there exists $N\in \Ob(\C)$ such that $P^{-2}$ and $P^{-1}$ are both generated in degrees $\leqslant N$.

Suppose $n\geqslant N$. Let $U$ be the submodule $\bigoplus_{m>n} W(m)$ of $W$. We have a short exact sequence
\begin{equation*}
0 \longrightarrow U \longrightarrow W \longrightarrow \jmath_*( \jmath^*(W) )\longrightarrow 0,
\end{equation*}
and hence a long exact sequence
\begin{multline*}
\cdots \to \Ext^1_{\C} (V,U) \to \Ext^1_{\C} (V, W) \to \Ext^1_{\C} (V,  \jmath_*( \jmath^*(W) ))
\to  \Ext^2_{\C}(V,W) \to \cdots .
\end{multline*}
But, for $i=1, 2$, one has $\Hom_{\C}(P^{-i}, U)=0$ and so $\Ext^i_{\C} (V,U)=0$. It follows that
\begin{equation*}
\Ext^1_{\C} (V, W) =  \Ext^1_{\C} (V,  \jmath_*( \jmath^*(W) )) = \Ext^1_{\C_n}(\jmath^*(V), \jmath^*(W)),
\end{equation*}
using Lemma \ref{eckmann-shapiro1}.
\end{proof}

\subsection{Injective resolutions of finite dimensional modules}
Suppose that $k$ is a field. We denote by $\D$ the standard duality functor $\Hom_k ( - , k)$ between the categories $\C_n\module$ and $\C_n^{\op}\module$. Any finite dimensional injective $\C_n$-module is isomorphic to $\D(P)$ for some finite dimensional projective $\C_n^{\op}$-module $P$.

\begin{lemma} \label{finite dimensional injective 1}
Suppose $k$ is a field. If $W$ is a finite dimensional injective $\C_n$-module for some $n\in \Ob(\C)$, then $\jmath_*(W)$ is a finite dimensional injective $\C$-module, where $\jmath: \C_n\hookrightarrow \C$.
\end{lemma}
\begin{proof}
It is obvious that $\jmath_*(W)$ is finite dimensional. The injectivity of $\jmath_*(W)$ follows from Lemma \ref{eckmann-shapiro1}.
\end{proof}

\begin{lemma} \label{inj res of finite dim module}
Suppose $k$ is a field of characteristic 0. Then every finite dimensional $\C$-module $V$ has a finite injective resolution
\begin{equation*}
 0 \to V  \to I^0 \to I^1 \to \cdots I^m \to 0
\end{equation*}
where $I^r$ is finite dimensional for each $r\in\{0, 1,\ldots, m\}$.
\end{lemma}
\begin{proof}
Choose $n\in\Ob(\C)$ such that $V = \jmath_*(\jmath^*(V))$ where $\jmath: \C_n \hookrightarrow \C$. It is easy to see that any finite dimensional $\C_n^{\op}$-module has a finite projective resolution in the category $\C_n^{\op}\module$. Using the functor $\D$, we deduce that $\jmath^*(V)$ has a finite injective resolution in the category $\C_n\module$.  By Lemma \ref{finite dimensional injective 1}, applying the functor $\jmath_*$ to this resolution gives a resolution of $V$ of the required form.
\end{proof}

\begin{remark}
Suppose $k$ is a field, and $V$ is a finite dimensional injective $\C$-module. It is easy to see that if $n\in \Ob(\C)$ and $V=\jmath_*(\jmath^*(V))$, where $\jmath: \C_n \hookrightarrow \C$ is the inclusion functor, then $\jmath^*(V)$ is an injective $\C_n$-module.
\end{remark}

\section{Injectivity of $k\C e_0$} \label{inj}

\subsection{Field of characteristic 0}
We say that the category $\C$ satisfies the \emph{transitivity condition} if for all $n\in \Ob(\C)$, the action of $G_{n+1}$ on $\C(n,n+1)$ is transitive. (Recall that $G_n$ denotes the group $\C(n,n)$; see Subsection \ref{notations and terminology}.)

\begin{lemma} \label{kC_n e_0 is injective}
Suppose that $k$ is a field of characteristic 0. If $\C$ satisfies the transitivity condition, and $0$ is an initial object of $\C$, then $k\C_n e_0$ is an injective $\C_n$-module for all $n\in \Ob(\C)$.
\end{lemma}

\begin{proof}
It suffices to prove that $\Ext^1_{\C_n} (k G_m, k\C_n e_0) = 0$ for all $m\in \Ob(\C_n)$. This is clear when $m=n$ for $k G_n$ is a projective $\C_n$-module.

Suppose $m<n$. Let $P=k\C_n e_m$ and let $U$ be the submodule $\bigoplus_{l>m} P(l)$ of $P$.
We have a short exact sequence
\begin{equation*}
0\longrightarrow U \longrightarrow P \longrightarrow k G_m  \longrightarrow 0,
\end{equation*}
and hence a long exact sequence
\begin{multline*}
0 \to \Hom_{\C_n}(k G_m, k\C_n e_0) \to \Hom_{\C_n} (P, k\C_n e_0) \to \Hom_{\C_n} (U, k\C_n e_0) \\
 \to \Ext^1_{\C_n} (k G_m, k\C_n e_0) \to \Ext^1_{\C_n} (P, k\C_n e_0) \to \cdots .
\end{multline*}
Since $m<n$, one has $\Hom_{\C_n}(k G_m, k\C_n e_0)=0$. Since $P$ is projective, $\Ext^1_{\C_n} (P, k\C_n e_0) =0$. Note that
\begin{equation*}
\dim_k \Hom_{k\C_n} (P, k\C_n e_0) = \dim_k k\C_n (0,m) = 1.
\end{equation*}
Since $U$ is generated by $U(m+1)$, and $G_{m+1}$ acts transitively on $U(m+1)$, one has
\begin{equation*}
\dim_k \Hom_{k\C_n} (U, k\C_n e_0) \leqslant \dim_k \Hom_{k G_{m+1}} (U(m+1), k\C_n (0,m+1)) \leqslant 1.
\end{equation*}
Hence, we must have $\Ext^1_{\C_n} (k G_m, k\C_n e_0) = 0$
\end{proof}

\begin{corollary} \label{kCe_0 is injective}
Suppose that $k$ is a field of characteristic 0. If $\C$ is locally Noetherian, satisfies the transitivity condition, and $0$ is an initial object of $\C$, then $k\C e_0$ is an injective $\C$-module.
\end{corollary}

\begin{proof}
By Corollary \ref{injectivity in C-mod}, it suffices to show that $k\C e_0$ is injective in $\C\module$. By Lemma \ref{reduction to C_n}, this follows from injectivity of $k\C_n e_0$ which is proved in Lemma \ref{kC_n e_0 is injective}.
\end{proof}

\subsection{Field of characteristic $p$}
Let us show that Corollary \ref{kCe_0 is injective} is false when $k$ is a field of characteristic $p>0$ and $\C$ is the category $\FI$.

For any $n\geqslant p$, we have a right action of the symmetric group $S_p$ on $\C(p,n)$. Let $U(n)$ be the subspace of $S_p$-invariant elements in $k\C(p,n)$, and let $U=\bigoplus_{n\geqslant p} U(n)$. Then $U$ is a submodule of $k\C e_p$. Now let $\J(n)$ be the set of $S_p$-orbits in $\C(p,n)$. For each orbit $J\in\J(n)$, let $\xi_J\in U(n)$ be the sum of all the $p!$ elements in $J$. Then the collection of $\xi_J$ for $J\in\J(n)$ is a basis for $U(n)$. There is a homomorphism $f: U \to k\C e_0$ such that $f(\xi_J)$ is the unique element of $\C(0,n)$ if $J\in \J(n)$. We define a submodule $W$ of $k\C e_0 \oplus k\C e_p$ as follows. For any $n\geqslant p$, let $W(n)$ be the set of all elements in $k\C(0,n)\oplus U(n)$ of the form $f(\xi) + \xi$ for $\xi \in U(n)$.

Let $V=(k\C e_0 \oplus k\C e_p)/W$, and let $\pi: k\C e_0 \oplus k\C e_p \to V$ be the canonical projection. Denote by $i: k\C e_0 \to V$ the restriction of $\pi$ to $k\C e_0$. It is clear that $i$ is a monomorphism. We claim that the short exact sequence
\begin{equation*}
0 \longrightarrow k\C e_0 \stackrel{i}{\longrightarrow} V \longrightarrow V/i(k\C e_0) \longrightarrow 0
\end{equation*}
does not split. Indeed, any homomorphism $k\C e_0 \oplus k\C e_p\to k\C e_0$ whose restriction to $k\C e_0$ is the identity map cannot vanish identically on $W$, for any homomorphism $k\C e_p \to k\C e_0$ must vanish identically on $U$.

\section{Restriction and coinduction along genetic functors}

Throughout this section, $k$ denotes any commutative ring.

\subsection{Restriction and coinduction}
A coinduction functor can be defined whenever one has a subring of a ring, and it is a right adjoint functor to the restriction functor; we refer the reader to \cite[Section 2.8]{Benson} for a clear exposition on the definition and basic properties of the coinduction functor in such a general setting. In this section, we begin our study of the coinduction functor in our special setting, where the pair of subring and the ring are isomorphic.

Let $\iota : \C \to \C$ be a faithful functor such that $\iota(n)=n+1$ for all $n\in\Ob(\C)$. We define the restriction functor $\s: \C\Module \longrightarrow \C\Module$ by $\s(V)=V\circ \iota$ for all $V\in\C\Module$; thus, $\s(V)(n)=V(n+1)$.

\begin{definition}
Suppose $V\in\C\Module$. We define $\q(V)\in\C\Module$ by
\begin{equation*}
\q(V)(n) = \Hom_\C (\s(k\C e_n), V) \quad \mbox{ for each }  n\in \Ob(\C).
\end{equation*}
We call $\q: \C\Module \longrightarrow \C\Module$ the \emph{coinduction} functor.
\end{definition}

Observe that any $\alpha \in \C(m,n)$ defines a $\C$-module homomorphism
\begin{equation*}
\s(k\C e_n) \longrightarrow \s(k\C e_m), \quad \gamma \mapsto \gamma\alpha.
\end{equation*}
The $\C$-module structure on $\q(V)$ is defined in the natural way as follows: if $\alpha\in \C(m,n)$ and $\varrho\in \q(V)(m)$, then $\alpha(\varrho)\in \q(V)(n)$ is the $\C$-module homomorphism
\begin{equation*}
\s(k\C e_n) \longrightarrow V, \quad \gamma \mapsto \varrho(\gamma\alpha).
\end{equation*}

\begin{lemma} \label{Q is right adjoint}
The functor $\q$ is right adjoint to the functor $\s$.
\end{lemma}
\begin{proof}
Let
\begin{equation*}
M = \bigoplus_{\substack{m\geqslant 0,\\n\geqslant 1}} k\C(m,n).
\end{equation*}
We have a $k\C$-bimodule structure on $M$ defined by
\begin{equation*}
\alpha \cdot \gamma = \iota(\alpha)\gamma,\qquad \gamma \cdot \alpha = \gamma\alpha,
\end{equation*}
for $\alpha \in k\C$ and $\gamma \in M$. By the tensor-hom adjunction, one has
\begin{equation*}
\Hom_{k\C} (M\otimes_{k\C} V, W) = \Hom_{k\C} (V, \Hom_{k\C} (M, W))
\end{equation*}
for any $V, W\in \C\Module$.

The $k\C$-module homomorphism
\begin{equation*}
M\otimes_{k\C} V \longrightarrow \s(V), \quad \gamma \otimes v \mapsto \gamma v
\end{equation*}
has an inverse defined on $\s(V)(n)$ by $v\mapsto e_{n+1}\otimes v$, for each $n\in \Ob(\C)$.
Hence, $M\otimes_{k\C} V$ is isomorphic to $\s(V)$.

On the other hand, there is a $k\C$-module direct sum decomposition
\begin{equation*}
M = \bigoplus_{m\geqslant 0} \s(k\C e_m),
\end{equation*}
so
\begin{equation*}
\Hom_{k\C} (M, W)  =  \prod_{m\geqslant 0} \Hom_{\C} (\s(k\C e_m), W).
\end{equation*}
But since $V$ is a \emph{graded} $k\C$-module, the image of any $k\C$-module homomorphism from $V$ to $\Hom_{k\C}(M,W)$ lies in $\q(W)$. It follows that
\begin{equation*}
 \Hom_{k\C} (V, \Hom_{k\C} (M, W)) = \Hom_\C (V, \q(W)).
\end{equation*}
Thus, $\Hom_\C(\s(V), W)  = \Hom_\C( V,\q(W))$.
\end{proof}

Following \cite{GL-Koszulity}, we call $\iota$ a \emph{genetic} functor if, for each $n\in \Ob(\C)$, the $\C$-module $\s(k\C e_n)$ is projective and generated in degrees $\leqslant n$.

\begin{lemma} \label{eckmann-shapiro2}
Suppose $\iota$ is a genetic functor. Then for any $V, W\in \C\Module$, one has
\begin{equation*}
\Ext^i_{\C} (\s(V), W) = \Ext^i_{\C} (V,\q(W))
\end{equation*}
for all $i\geqslant 1$.
\end{lemma}
\begin{proof}
This is immediate from Lemma \ref{Q is right adjoint} and the Eckmann-Shapiro lemma.
\end{proof}

\begin{remark}
The condition in Lemma \ref{eckmann-shapiro2} that $\iota$ is a genetic functor can be weakened. Indeed, we only need to use the property that for each $n\in \Ob(\C)$, the $\C$-module $\s(k\C e_n)$ is projective.
\end{remark}

\subsection{Genetic functors for $\FI_G$ and $\VI$}
Suppose $\C$ is $\FI_G$ or $\VI$.  There is a natural monoidal structure $\odot$ on $\C$ such that $m\odot n=m+n$  for all $m, n \in \Ob(\C)$. Let us recall this monoidal structure.

{\bf Case 1}: Suppose $\C$ is $\FI_G$. For any $(f_1, c_1)\in \C(m_1,n_1)$ and $(f_2, c_2)\in \C(m_2,n_2)$, we define $(f_1, c_1) \odot (f_2, c_2)$ to be the morphism $(f,c)\in \C(m_1+m_2, n_1+n_2)$ where
\begin{gather*}
f(t) = \left\{  \begin{array}{ll}
f_1(t) & \mbox{ if } t\leqslant m_1,\\
f_2(t-m_1)+n_1 & \mbox{ if } t> m_1.
\end{array}\right.
\end{gather*}
and
\begin{gather*}
c(t) = \left\{  \begin{array}{ll}
c_1(t) & \mbox{ if } r\leqslant m_1,\\
c_2(t-m_1) & \mbox{ if } t> m_1.
\end{array}\right.
\end{gather*}

{\bf Case 2}: Suppose $\C$ is $\VI$. For any $f_1\in \C(m_1, n_1)$ and $f_2 \in \C(m_2,n_2)$, we define $f_1 \odot f_2 \in \C(m_1+m_2, n_1+ n_2)$ by
\begin{equation*}
f_1 \odot f_2 = f_1 \oplus f_2 : \F^{m_1}\oplus \F^{m_2} \longrightarrow \F^{n_1}\oplus \F^{n_2}.
\end{equation*}

In both cases, we let $\iota: \C \to \C$ be the functor defined by (\ref{iota}). It is clear that $\iota$ is faithful. In the next two sections, we shall briefly recall the proof that $\iota$ is a genetic functor, and examine the structure of $\q(k\C e_m)$ for each $m\in \Ob(\C)$.

\section{Structure of $\q(k\C e_m)$ when $\C$ is $\FI_G$} \label{FI_G section}

Throughout this section, $k$ denotes any commutative ring.

\subsection{Structure of $\s(k\C e_n)$}  \label{structure of S for FI_G}
Suppose that $\C$ is $\FI_G$. Let $n\in \Ob(\C)$. We now recall the structure of $\s(k\C e_n)$.

Denote by $e$ the identity element of $G$, and define the morphisms
\begin{equation*}
(f_n,, c_n) \in \C(n,n+1) \quad\mbox{ and }\quad (f_{n,r,g}, c_{n,r,g}) \in \C(n, n) \quad \mbox{ for } r\in [n],\, g\in G,
\end{equation*}
by
\begin{align*}
f_n(t) &= t+1,\\
 c_n(t) &= e,\\
f_{n,r,g}(t) &= \left\{ \begin{array}{ll}
t+1 & \mbox{ if }t<r,\\
1 & \mbox{ if } t=r,\\
t & \mbox{ if }t>r,
\end{array} \right. \\
c_{n,r,g}(t) &=  \left\{ \begin{array}{ll}
e & \mbox{ if }t \neq r,\\
g & \mbox{ if } t=r,
\end{array} \right.
\end{align*}
for $t\in [n]$.

Now, for any $l \in \Ob(\C)$, $r\in [n]$, $g\in G$, define the maps
\begin{gather*}
\Phi_{n,0} : \C(n, l) \longrightarrow \C(n,l+1), \quad \alpha\mapsto \iota(\alpha)\circ (f_n,c_n);\\
\Phi_{n,r,g} : \C(n-1,l) \longrightarrow \C(n,l+1), \quad \alpha\mapsto \iota(\alpha)\circ (f_{n,r,g},c_{n,r,g}).
\end{gather*}
We may extend these maps linearly to $\C$-module homomorphisms
\begin{equation*}
\Phi_{n,0} : k\C e_n \longrightarrow \s(k\C e_n) \quad \mbox{ and } \quad \Phi_{n,r,g} : k\C e_{n-1}\longrightarrow \s(k\C e_n).
\end{equation*}
Let
\begin{equation} \label{phi for FI_G}
\Phi_n : k\C e_n \oplus \left( \bigoplus_{r\in [n],\, g\in G} k\C e_{n-1} \right) \longrightarrow \s(k\C e_n)
\end{equation}
be the $\C$-module homomorphism whose restriction to $k\C e_n$ is $\Phi_{n,0}$ and whose restriction to the direct summand $k\C e_{n-1}$ indexed by $r\in [n]$, $g\in G$ is $\Phi_{n,r,g}$. It is straightforward to verify that $\Phi_n$ is an isomorphism (see \cite[Section 5]{GL-Koszulity}).
Thus, $\iota$ is a genetic functor.

\subsection{Preliminary discussion of $\q(k\C e_m)$} \label{preliminary discussion 1}
We retain the notations of subsection \ref{structure of S for FI_G}. Let $m, n \in \Ob(\C)$. By the isomorphism (\ref{phi for FI_G}), one has the identification
\begin{equation*}
\q(k \C e_m)(n) = \Hom_{\C}( k\C e_n, k\C e_m ) \oplus \left( \bigoplus_{r\in [n],\, g\in G} \Hom_{\C} ( k\C e_{n-1}, k\C e_m) \right).
\end{equation*}
Denote by
\begin{align*}
\Psi_{n,0}: k\C(m,n) &\longrightarrow \q(k\C e_m)(n),\\
\Psi_{n,r,g} : k\C(m,n-1) &\longrightarrow \q(k\C e_m)(n)
\end{align*}
the linear maps where $\Psi_{n,0}$ is the natural bijection of $k\C(m,n)$ with the direct summand $\Hom_{\C} (k\C e_n, k\C e_m)$ of $\q(k\C e_m)(n)$, and $\Psi_{n,r,g}$ is the natural bijection of $k\C(m,n-1)$ with the direct summand $\Hom_{\C}(k\C e_{n-1}, k\C e_m)$ of $\q(k\C e_m)(n)$ indexed by $r\in [n]$, $g\in G$. We have a linear bijection
\begin{equation*}
\Psi_n: k\C(m,n) \oplus \left( \bigoplus_{r\in [n],\, g\in G} k\C(m,n-1) \right) \longrightarrow \q(k\C e_m)(n)
\end{equation*}
whose restriction to $k\C(m,n)$ is $\Psi_{n,0}$ and whose restriction to the direct summand $k\C(m,n-1)$ indexed by $r\in [n]$, $g\in G$ is $\Psi_{n,r,g}$.

The next lemma describes the $\C$-module structure of $\q(k\C e_m)$ in terms of the identifications $\Psi_n$ for $n\in \Ob(\C)$. We shall use the following notations. For any $l \geqslant 1$ and $r\in [l]$, let $\partial_r : [l]\setminus\{r\} \to [l-1]$ be the unique nondecreasing bijection. If  $\alpha = (f,c) \in \C(n,l)$, $r\in [l]\setminus \im(f)$, and $s\in [n]$, we let
\begin{align*}
\partial_r \alpha &= (\partial_r\circ f, c) \in \C(n,l-1), \\
 \alpha_s &= (\partial_{f(s)}\circ f \circ \partial_s^{-1} , c\circ \partial_s^{-1} )\in \C(n-1,l-1),
\end{align*}
where $\partial_s^{-1} : [n-1]\to [n]\setminus \{ s \}$ is the inverse map of $\partial_s$.

\begin{lemma} \label{preliminary structure 1}
Suppose that $\C$ is $\FI_G$.

(i) Let $\alpha =(f,c) \in \C(n,l)$ and $\beta \in \C(m,n)$. Then
\begin{equation*}
\alpha \Psi_{n,0}(\beta) = \Psi_{l,0}(\alpha\beta) + \sum_{\substack{r\in [l]\setminus \im(f),\\ g\in G } } \Psi_{l, r, g} (\partial_r \alpha \beta).
\end{equation*}
(ii) Let $\alpha = (f,c) \in \C(n,l)$, $\beta \in \C(m,n-1)$, $s \in [n]$, and $h \in G$. Then
\begin{equation*}
\alpha \Psi_{n,s,h}(\beta) = \Psi_{l, f(s), h\cdot c(s)^{-1}} (\alpha_s \beta).
\end{equation*}
\end{lemma}
\begin{proof}
(i) Suppose $\gamma \in \C(l,i)$. Then
\begin{align*}
\alpha \Psi_{n,0}(\beta) (\Phi_{l,0}(\gamma) ) &= \Psi_{n,0}(\beta) ( \iota(\gamma)\circ (f_l,c_l)\circ \alpha )\\
&= \Psi_{n,0}(\beta) ( \iota(\gamma)\circ  \iota(\alpha) \circ  (f_n,c_n) )\\
&= \Psi_{n,0}(\beta) ( \Phi_{ n, 0}  (\gamma \alpha) )\\
&= \gamma\alpha\beta\\
&=  \Psi_{l,0}(\alpha\beta) (\Phi_{l,0}(\gamma)).
\end{align*}

Suppose $\gamma \in \C(l-1,i)$, $r\in [l]\setminus \im(f)$, and $g\in G$. Then
\begin{align*}
\alpha \Psi_{n,0}(\beta) (\Phi_{l,r,g}(\gamma) ) &= \Psi_{n,0}(\beta) ( \iota(\gamma)\circ (f_{l,r,g},c_{l,r,g})\circ \alpha )\\
&=  \Psi_{n,0}(\beta) ( \iota(\gamma)\circ  \iota(\partial_r \alpha ) \circ  (f_n,c_n) )\\
&= \Psi_{n,0}(\beta) (\Phi_{n,0}(\gamma \partial_r \alpha) ) \\
&= \gamma \partial_r \alpha \beta \\
&= \Psi_{l, r, g} (\partial_r \alpha \beta) (\Phi_{l,r,g}(\gamma) ).
\end{align*}

Suppose $\gamma \in \C(l-1,i)$, $r\in \im(f)$, and $g\in G$. Then
\begin{align*}
\alpha \Psi_{n,0}(\beta) (\Phi_{l,r,g}(\gamma) ) &= \Psi_{n,0}(\beta) ( \iota(\gamma)\circ (f_{l,r,g},c_{l,r,g})\circ \alpha )\\
&= 0.
\end{align*}

(ii)  Suppose $\gamma \in \C(l,i)$. Then
\begin{align*}
\alpha \Psi_{n,s,h}(\beta) (\Phi_{l,0}(\gamma) )   &=  \Psi_{n,s,h}(\beta)(  \iota(\gamma)\circ (f_l,c_l)\circ \alpha ) \\
&=  \Psi_{n,s,h}(\beta)( \iota(\gamma)\circ  \iota(\alpha) \circ  (f_n,c_n) )\\
&= 0.
\end{align*}

Suppose $\gamma \in \C(l-1,i)$, $r\in [l]$, and $g\in G$.

If $r = f(s)$ and $g \cdot c(s) = h$, then
\begin{align*}
\alpha \Psi_{n,s,h}(\beta) (\Phi_{l,r,g}(\gamma) ) &= \Psi_{n,s,h}(\beta)( \iota(\gamma)\circ (f_{l,r,g},c_{l,r,g})\circ \alpha  )\\
&= \Psi_{n,s,h}(\beta)( \iota(\gamma)\circ \iota( \alpha_s ) \circ (f_{n,s,h}, c_{n,s,h} ) ) \\
&= \Psi_{n,s,h}(\beta)(\Phi_{n,s,h}(\gamma \alpha_s) )\\
&= \gamma \alpha_s\beta \\
&= \Psi_{l, f(s), h\cdot c(s)^{-1}} (\alpha_s \beta) ( \Phi_{l,r,g}(\gamma)  )
\end{align*}

If $r \neq f(s)$ or $g \cdot c(s) \neq h$, then
\begin{align*}
\alpha \Psi_{n,s,h}(\beta) (\Phi_{l,r,g}(\gamma) ) &= \Psi_{n,s,h}(\beta)( \iota(\gamma)\circ (f_{l,r,g},c_{l,r,g})\circ \alpha  )\\
&= 0.
\end{align*}
\end{proof}

\subsection{Proof of Theorem \ref{main result on q 1}}
We retain the notations of subsection \ref{preliminary discussion 1}.

\begin{proof}[Proof of Theorem \ref{main result on q 1}]
Let $\C$ be $\FI_G$. Let $m\in \Ob(\C)$. We need to prove that $\q(k\C e_m)$ is isomorphic to $k\C e_m \oplus k\C e_{m+1}$.

For each $n\in \Ob(\C)$, let $U(n)$ be the image of $\displaystyle\bigoplus_{r\in [n],\, g\in G} k\C(m,n-1)$ in $\q(k \C e_m)(n)$ under $\Psi_n$, and let $U = \bigoplus_{n\in \Ob(\C)} U(n)$. By Lemma \ref{preliminary structure 1}, $U$ is a $\C$-submodule of $\q(k\C e_m)$ and there is a short exact sequence
\begin{equation*}
0 \longrightarrow U \longrightarrow \q(k\C e_m) \longrightarrow k\C e_m \longrightarrow 0.
\end{equation*}
Since $k\C e_m$ is projective, this short exact sequence splits. It suffices to show that $U$ is isomorphic to $k\C e_{m+1}$.

For each $n\in \Ob(\C)$, define a linear map $\Theta_n: U(n) \to k\C(m+1, n)$ by
\begin{equation*}
 \Theta_n ( \Psi_{n,s,h} (\beta) ) = (f_{n,s,h}, c_{n,s,h})^{-1} \iota(\beta)
\end{equation*}
for all $s\in [n]$, $h\in G$, and $\beta \in \C(m,n-1)$. Let $\Theta: U \to k\C e_{m+1}$ be the linear map whose restriction to $U(n)$ is $\Theta_n$. It is easy to see that $\Theta$ is bijective.

We show now that $\Theta$ is a $\C$-module homomorphism. Suppose $\alpha =(f,c)\in \C(n,l)$, $\beta\in \C(m,n-1)$, $s\in [n]$, and $h\in G$. Observe that
\begin{equation*}
\iota(\alpha_s) \circ (f_{n,s,h}, c_{n,s,h}) = (f_{l, f(s), h\cdot c(s)^{-1}}, c_{l, f(s), h\cdot c(s)^{-1}}) \circ \alpha.
\end{equation*}
Hence,
\begin{align*}
\Theta_l (\alpha \Psi_{n,s,h}(\beta)) &= \Theta_l(\Psi_{l, f(s), h\cdot c(s)^{-1}} (\alpha_s \beta)) \\
&= (f_{l, f(s), h\cdot c(s)^{-1}}, c_{l, f(s), h\cdot c(s)^{-1}} )^{-1} \iota(\alpha_s \beta)\\
&= \alpha(f_{n,s,h}, c_{n,s,h})^{-1} \iota(\beta) \\
&= \alpha \Theta_n ( \Psi_{n,s,h} (\beta) ).
\end{align*}
\end{proof}

\section{Structure of $\q(k\C e_m)$ when $\C$ is $\VI$} \label{VI section}

Throughout this section, $k$ denotes any commutative ring.

\subsection{Structure of $\s(k\C e_n)$}  \label{structure of S for VI}
Suppose that $\C$ is $\VI$. Let $n\in \Ob(\C)$. We now recall the structure of $\s(k\C e_n)$.

We write elements $\alpha\in \C(n,l)$ as a $l\times n$-matrix. We write elements $u\in \F^n$ as a column vector and $u^t$ for its transpose. Let $\PP(\F^n)$ be the set of one dimensional vector subspaces of $\F^n$. For any $u\in \F^n$ and $\ell\in \PP(\F^n)$, we write $u^t(\ell)\neq 0$ if $u^t v \neq 0$ for any nonzero vector $v$ in $\ell$. For each $\ell \in \PP(\F^n)$, we choose and fix a $(n-1)\times n$-matrix $\varpi_\ell : \F^n \to \F^{n-1}$ whose kernel is $\ell$. We shall denote identity matrices by $I$.

Now, for any $u\in \F^n$ and $\ell \in \PP(\F^n)$ such that $u^t(\ell)\neq 0$, define the maps
\begin{gather*}
\Phi_{n,u,0} : \C(n, l) \longrightarrow \C(n,l+1), \quad \alpha\mapsto
\left( \begin{matrix} 1& 0 \\ 0 & \alpha \end{matrix} \right) \left( \begin{matrix}  u^t \\ I \end{matrix} \right) ; \\
\Phi_{n,u,\ell} : \C(n-1,l) \longrightarrow \C(n,l+1), \quad \alpha\mapsto
\left( \begin{matrix} 1& 0 \\ 0 & \alpha \end{matrix} \right) \left( \begin{matrix}  u^t \\ \varpi_\ell \end{matrix} \right) .
\end{gather*}
We may extend these maps linearly to $\C$-module homomorphisms
\begin{equation*}
\Phi_{n,u,0} : k\C e_n \longrightarrow \s(k\C e_n) \quad \mbox{ and } \quad \Phi_{n,u,\ell} : k\C e_{n-1}\longrightarrow \s(k\C e_n),
\end{equation*}
Let
\begin{equation} \label{phi for VI}
\Phi_n : \left( \bigoplus_{u\in \F^n} k\C e_n \right) \oplus \left( \bigoplus_{u\in \F^n} \bigoplus_{\substack{\ell \in \PP(\F^n) \\ u^t(\ell)\neq 0 }} k\C e_{n-1} \right) \longrightarrow \s(k\C e_n)
\end{equation}
be the $\C$-module homomorphism whose restriction to the direct summand $k\C e_n$ indexed by $u\in \F^n$ is $\Phi_{n,u,0}$ and whose restriction to the direct summand $k\C e_{n-1}$ indexed by $u\in \F^n$, $\ell\in \PP(\F^n)$ is $\Phi_{n,u,\ell}$. It is straightforward to verify that $\Phi_n$ is an isomorphism (see \cite[Section 5]{GL-Koszulity}).
Thus, $\iota$ is a genetic functor.

\subsection{Preliminary discussion of $\q(k\C e_m)$} \label{preliminary discussion 2}
We retain the notations of subsection \ref{structure of S for VI}. Let $m, n \in \Ob(\C)$. By the isomorphism (\ref{phi for VI}), one has the identification
\begin{equation*}
\q(k \C e_m)(n) = \left( \bigoplus_{u\in \F^n} \Hom_{\C}( k\C e_n, k\C e_m ) \right) \oplus  \left( \bigoplus_{u\in \F^n} \bigoplus_{\substack{\ell \in \PP(\F^n) \\ u^t(\ell)\neq 0 }}  \Hom_{\C} ( k\C e_{n-1}, k\C e_m) \right).
\end{equation*}
Denote by
\begin{align*}
\Psi_{n,u,0}: k\C(m,n) &\longrightarrow \q(k\C e_m)(n),\\
\Psi_{n,u,\ell} : k\C(m,n-1) &\longrightarrow \q(k\C e_m)(n)
\end{align*}
the linear maps where $\Psi_{n,u,0}$ is the natural bijection of $k\C(m,n)$ with the direct summand $\Hom_{\C} (k\C e_n, k\C e_m)$ of $\q(k\C e_m)(n)$ indexed by $u\in \F^n$, and $\Psi_{n,u,\ell}$ is the natural bijection of $k\C(m,n-1)$ with the direct summand $\Hom_{\C}(k\C e_{n-1}, k\C e_m)$ of $\q(k\C e_m)(n)$ indexed by $u\in \F^n$, $\ell\in \PP(\F^n)$. We have a linear bijection
\begin{equation*}
\Psi_n:\left( \bigoplus_{u\in \F^n} k\C (m,n) \right) \oplus  \left( \bigoplus_{u\in \F^n} \bigoplus_{\substack{\ell \in \PP(\F^n) \\ u^t(\ell)\neq 0 }} k\C(m, n-1) \right)\longrightarrow \q(k\C e_m)(n)
\end{equation*}
whose restriction to the direct summand $k\C(m,n)$ indexed by $u\in\F^n$ is $\Psi_{n,u,0}$ and whose restriction to the direct summand $k\C(m,n-1)$ indexed by $u\in \F^n$, $\ell \in \PP(\F^n)$ is $\Psi_{n,u,\ell}$.

The next lemma describes the $\C$-module structure of $\q(k\C e_m)$ in terms of the identifications $\Psi_n$ for $n\in \Ob(\C)$. We shall use the following notation. For any $\alpha\in \C(n,l)$ and $\wp\in \PP(\F^n)$, let $\alpha_{\wp} \in \C(n-1,l-1)$ be the unique linear map such that the following diagram commutes:
\begin{equation*}
\xymatrix{  \F^n \ar[r]^{\alpha} \ar[d]_{\varpi_{\wp}} & \F^l \ar[d]^{\varpi_{\alpha(\wp)} } \\
\F^{n-1} \ar[r]_{\alpha_{\wp}} & \F^{l-1} }
\end{equation*}

\begin{lemma} \label{preliminary structure 2}
Suppose  that $\C$ is $\VI$.

(i) Let $\alpha\in \C(n,l)$, $\beta\in \C(m,n)$, and $v\in \F^n$. Then
\begin{equation*}
\alpha \Psi_{n,v,0}(\beta) = \sum_{\substack{u \in \F^l \\ u^t \alpha = v^t}} \Psi_{l,u,0}(\alpha\beta) + \sum_{\substack{u \in \F^l \\ u^t \alpha = v^t}} \sum_{ \substack{\ell \in \PP(\F^l) \\ u^t(\ell)\neq 0 \\  \ell \nsubseteq \im(\alpha) } } \Psi_{l,u,\ell} ( \varpi_\ell \alpha \beta )  .
\end{equation*}

(ii) Let $\alpha\in \C(n,l)$, $\beta\in \C(m,n-1)$, $v\in \F^n$, and $\wp\in \PP(\F^n)$. Suppose $v^t(\wp) \neq 0$. Then
\begin{equation*}
\alpha \Psi_{n,v,\wp} (\beta) = \sum_{\substack{u \in \F^l \\ u^t \alpha = v^t}} \sum_{ \substack{\ell \in \PP(\F^l)\\ \ell = \alpha(\wp) } }  \Psi_{l,u,\ell}( \alpha_{\wp} \beta )  .
\end{equation*}
\end{lemma}

\begin{proof}
(i) Suppose $\gamma \in \C(l,i)$ and $u\in \F^l$. One has
\begin{equation} \label{preliminary structure 2 equation 1}
\Phi_{l,u,0}(\gamma)\alpha
= \left( \begin{matrix} 1& 0 \\ 0 & \gamma \end{matrix} \right) \left( \begin{matrix} u^t \alpha \\ \alpha \end{matrix}\right)
= \left( \begin{matrix} 1& 0 \\ 0 & \gamma \alpha \end{matrix}  \right) \left( \begin{matrix} u^t \alpha \\ I \end{matrix} \right).
\end{equation}
Thus,
\begin{equation*}
\alpha \Psi_{n,v,0}(\beta) (\Phi_{l,u,0}(\gamma)) =  \Psi_{n,v,0}(\beta) (\Phi_{l,u,0}(\gamma)\alpha)
= \left\{  \begin{array}{ll}
0 & \mbox{ if } u^t\alpha \neq v^t, \\
\gamma\alpha\beta & \mbox{ if } u^t\alpha = v^t.
\end{array} \right.
\end{equation*}
In particular, when $u^t\alpha=v^t$, one has
\begin{equation*}
\alpha \Psi_{n,v,0}(\beta) (\Phi_{l,u,0}(\gamma)) =  \Psi_{l,u,0}(\alpha\beta)(\Phi_{l,u,0}(\gamma)).
\end{equation*}

Now suppose $\gamma\in \C(l-1,i)$, $u\in \F^l$, $\ell\in \PP(\F^l)$, and $u^t(\ell)\neq 0$. One has
\begin{equation*}
\Phi_{l,u,\ell}(\gamma)\alpha = \left( \begin{matrix} 1& 0 \\ 0 & \gamma \end{matrix} \right) \left( \begin{matrix} u^t \alpha \\  \varpi_\ell \alpha \end{matrix}\right)
=  \left( \begin{matrix} 1& 0 \\ 0 & \gamma \varpi_\ell \alpha \end{matrix}  \right) \left( \begin{matrix} u^t \alpha \\ I \end{matrix} \right).
\end{equation*}
Thus,
\begin{align*}
\alpha \Psi_{n,v,0}(\beta) (\Phi_{l,u,\ell}(\gamma)) &= \Psi_{n,v,0}(\beta) (\Phi_{l,u,\ell}(\gamma)\alpha) \\
&= \left\{  \begin{array}{ll}
0 & \mbox{ if } u^t\alpha \neq v^t \mbox{ or } \ell \subseteq \im(\alpha), \\
\gamma \varpi_\ell \alpha\beta & \mbox{ if } u^t\alpha = v^t \mbox{ and } \ell \nsubseteq \im(\alpha).
\end{array} \right.
\end{align*}
In particular, when $ u^t\alpha = v^t$ and $\ell \nsubseteq \im(\alpha)$, one has
\begin{equation*}
\alpha \Psi_{n,v,0}(\beta) (\Phi_{l,u,\ell}(\gamma)) = \Psi_{l,u,\ell} ( \varpi_\ell \alpha \beta ) (\Phi_{l,u,\ell}(\gamma)).
\end{equation*}

(ii)  Suppose $\gamma \in \C(l,i)$ and $u\in \F^l$. From (\ref{preliminary structure 2 equation 1}), one has
\begin{equation*}
\alpha \Psi_{n,v,\wp} (\beta)(\Phi_{l,u,0}(\gamma)) = \Psi_{n,v,\wp} (\beta) (\Phi_{l,u,0}(\gamma)\alpha) =0.
\end{equation*}

Now suppose $\gamma\in \C(l-1,i)$, $u\in \F^l$, $\ell\in \PP(\F^l)$, and $u^t(\ell)\neq 0$. One has
\begin{equation*}
\Phi_{l,u,\ell}(\gamma)\alpha = \left( \begin{matrix} 1& 0 \\ 0 & \gamma \end{matrix} \right) \left( \begin{matrix} u^t \alpha \\  \varpi_\ell \alpha \end{matrix}\right).
\end{equation*}
We can write $\left( \begin{matrix} u^t \alpha \\ \varpi_\ell \alpha \end{matrix}\right)$ in the form $ \left( \begin{matrix} 1& 0 \\ 0 & * \end{matrix} \right)  \left( \begin{matrix} u^t \alpha \\ \varpi_{\wp} \end{matrix}\right)$ if and only if $\ell = \alpha(\wp)$. If $\ell = \alpha(\wp)$, then
\begin{equation*}
\Phi_{l,u,\ell}(\gamma)\alpha =  \left( \begin{matrix} 1& 0 \\ 0 & \gamma \alpha_{\wp} \end{matrix} \right)  \left( \begin{matrix} u^t \alpha \\ \varpi_{\wp} \end{matrix}\right).
\end{equation*}
Thus,
\begin{align*}
\alpha\Psi_{n,v,\wp} (\beta) \Phi_{l,u,\ell}(\gamma) &= \Psi_{n,v,\wp} (\beta) (\Phi_{l,u,\ell}(\gamma)\alpha) \\
&= \left\{  \begin{array}{ll}
0 & \mbox{ if } u^t\alpha \neq v^t \mbox{ or } \ell \neq \alpha(\wp), \\
\gamma  \alpha_{\wp} \beta & \mbox{ if } u^t\alpha = v^t \mbox{ and } \ell = \alpha(\wp).
\end{array} \right.
\end{align*}
In particular, when $u^t\alpha = v^t$ and $ \ell = \alpha(\wp)$,  one has
\begin{equation*}
\alpha\Psi_{n,v,\wp} (\beta) (\Phi_{l,u,\ell}(\gamma)  ) =  \Psi_{l,u,\ell}( \alpha_{\wp} \beta ) (\Phi_{l,u,\ell}(\gamma)).
\end{equation*}
\end{proof}

\subsection{Proof of Theorem \ref{main result on q 2}}
We retain the notations of subsection \ref{preliminary discussion 2}.

\begin{proof}[Proof of Theorem \ref{main result on q 2}]
Let $\C$ be $\VI$. Let $m\in \Ob(\C)$. We need to prove that $\q(k\C e_m)$ contains a direct summand isomorphic to $k\C e_{m+1}$.

Since $k\C e_{m+1}$ is a projective $\C$-module, it suffices to construct a surjective homomorphism
\begin{equation*}
\pi : \q(k\C e_m) \longrightarrow k\C e_{m+1}.
\end{equation*}
For each $n\in \Ob(\C)$, we define a linear map $\pi_n : \q(k\C e_m)(n) \longrightarrow k\C(m+1, n)$ as follows:
\begin{enumerate}
\item
If $\beta\in \C(m,n)$ and $v\in \F^n$, let
\begin{equation*}
\pi_n (\Psi_{n,v,0}(\beta) ) = -q^{-n} \sum_{ \substack{ \wp \in \PP(\F^n) \\ v^t(\wp)\neq 0 \\  \wp \nsubseteq \im(\beta) } } \left( \begin{matrix} v^t \\ \varpi_{\wp}  \end{matrix} \right)^{-1} \left( \begin{matrix} 1 & 0 \\ 0 & \varpi_{\wp} \beta \end{matrix} \right).
\end{equation*}

\item
If $\beta\in \C(m,n-1)$, $v\in \F^n$, $\wp \in \PP(\F^n)$, and $v^t(\wp)\neq 0$, let
\begin{equation*}
\pi_n (\Psi_{n,v, \wp}(\beta) )= q^{-n} \left( \begin{matrix} v^t \\ \varpi_{\wp}  \end{matrix} \right)^{-1} \left( \begin{matrix} 1 & 0 \\ 0 & \beta \end{matrix} \right).
\end{equation*}
\end{enumerate}
Let $\pi$ be the linear map whose restriction to $\q(k\C e_m)(n)$ is $\pi_n$. We claim that:  $\pi$ is surjective, and $\pi$ is a $\C$-module homomorphism.

To show that $\pi$ is surjective, consider any $\gamma \in \C(m+1, n)$. We want to show that $\gamma$ is in the image of $\pi_n$. To this end, we write the $m+1$ columns of $\gamma$ as $\gamma_1, \ldots, \gamma_{m+1} \in \F^n$. Let $\wp \in \PP(\F^n)$ be the span of $\gamma_1$; so one has $\varpi_{\wp} \gamma_1=0$. Since $\gamma: \F^{m+1} \to \F^n$ is injective, its transpose $\gamma^t : \F^n \to \F^{m+1}$ is surjective. Hence, there exists $v\in \F^n$ such that
\begin{equation*}
v^t \gamma_1 = 1, \quad v^t \gamma_2 = \cdots v^t \gamma_{m+1} = 0.
\end{equation*}
Choose such a $v$. Then one has
\begin{equation*}
\left( \begin{matrix}  v^t \\ \varpi_{\wp} \end{matrix} \right) \gamma =  \left( \begin{matrix} 1 & 0 \\ 0 & \beta  \end{matrix} \right)
\end{equation*}
for some $\beta\in \C(m,n-1)$. Therefore,
\begin{equation*}
\pi_n( q^n \Psi_{n,v, \wp}(\beta) ) = \left( \begin{matrix} v^t \\ \varpi_{\wp}  \end{matrix} \right)^{-1} \left( \begin{matrix} 1 & 0 \\ 0 & \beta \end{matrix} \right) = \gamma.
\end{equation*}

It remains to check that $\pi$ is a $\C$-module homomorphism. Let $\alpha\in \C(n,l)$.

Observe that if $v\in \F^n$, then
\begin{equation*}
\# \{ u\in \F^l \mid u^t\alpha = v^t \} = q^{l-n}.
\end{equation*}

Suppose that $\beta\in \C(m,n)$ and $v\in \F^n$. One has:

\begin{align*}
& \pi_l (  \alpha  \Psi_{n,v,0}(\beta) ) \\
=& \sum_{\substack{u \in \F^l \\ u^t \alpha = v^t}} \pi_l( \Psi_{l,u,0}(\alpha\beta) ) + \sum_{\substack{u \in \F^l \\ u^t \alpha = v^t}} \sum_{ \substack{\ell \in \PP(\F^l) \\ u^t(\ell)\neq 0 \\  \ell \nsubseteq \im(\alpha) } } \pi_l( \Psi_{l,u,\ell} ( \varpi_\ell \alpha \beta )  ) \\
=& - q^{-l}  \sum_{\substack{u \in \F^l \\ u^t \alpha = v^t}}  \sum_{ \substack{ \ell \in \PP(\F^l) \\ u^t (\ell)\neq 0 \\ \ell \nsubseteq \im(\alpha \beta) } } \left( \begin{matrix} u^t \\ \varpi_\ell  \end{matrix} \right)^{-1} \left( \begin{matrix} 1 & 0 \\ 0 & \varpi_\ell \alpha \beta \end{matrix} \right) \\
&\quad + q^{-l}  \sum_{\substack{u \in \F^l \\ u^t \alpha = v^t}} \sum_{ \substack{\ell \in \PP(\F^l) \\ u^t (\ell)\neq 0 \\ \ell \nsubseteq \im(\alpha) } } \left( \begin{matrix} u^t \\ \varpi_\ell  \end{matrix} \right)^{-1} \left( \begin{matrix} 1 & 0 \\ 0 & \varpi_\ell \alpha \beta \end{matrix} \right) \\
=& - q^{-l}  \sum_{\substack{u \in \F^l \\ u^t \alpha = v^t}}  \sum_{ \substack{ \ell \in \PP(\F^l) \\ u^t (\ell)\neq 0 \\ \ell \subseteq \im(\alpha) \setminus \im(\alpha \beta) } } \left( \begin{matrix} u^t \\ \varpi_\ell  \end{matrix} \right)^{-1} \left( \begin{matrix} 1 & 0 \\ 0 & \varpi_\ell \alpha \beta \end{matrix} \right) \\
=& - q^{-l}  \sum_{\substack{u \in \F^l \\ u^t \alpha = v^t}}  \sum_{ \substack{ \wp \in \PP(\F^n) \\ v^t (\wp)\neq 0 \\ \wp \nsubseteq \im(\beta) } } \left( \begin{matrix} u^t \\ \varpi_{\alpha(\wp)}  \end{matrix} \right)^{-1} \left( \begin{matrix} 1 & 0 \\ 0 & \varpi_{\alpha(\wp)} \alpha \beta \end{matrix} \right) \\
=& - q^{-l}  \sum_{\substack{u \in \F^l \\ u^t \alpha = v^t}}  \sum_{ \substack{ \wp \in \PP(\F^n) \\ v^t (\wp)\neq 0 \\  \wp \nsubseteq \im(\beta) } } \left( \begin{matrix} u^t \\  \varpi_{\alpha(\wp)}  \end{matrix} \right)^{-1} \left( \begin{matrix} 1 & 0 \\ 0 &  \alpha_{\wp} \varpi_{\wp} \beta \end{matrix} \right).
\end{align*}
Observe that when $u^t\alpha = v^t$, one has
\begin{equation*}
 \left( \begin{matrix} u^t \\  \varpi_{\alpha(\wp)}  \end{matrix} \right) \alpha =  \left( \begin{matrix} v^t \\ \alpha_{\wp} \varpi_{\wp}   \end{matrix} \right) =  \left( \begin{matrix} 1 & 0 \\ 0 &  \alpha_{\wp} \end{matrix} \right)   \left( \begin{matrix} v^t \\  \varpi_{\wp}   \end{matrix} \right)
 \end{equation*}
which implies
\begin{equation} \label{key result equation}
 \left( \begin{matrix} u^t \\ \varpi_{\alpha(\wp)}  \end{matrix} \right)^{-1}  \left( \begin{matrix} 1 & 0 \\ 0 &  \alpha_{\wp} \end{matrix} \right)  = \alpha  \left( \begin{matrix} v^t \\  \varpi_{\wp}   \end{matrix} \right)^{-1}.
 \end{equation}
Hence, continuing our calculation from above,
\begin{align*}
& \pi_l (  \alpha  \Psi_{n,v,0}(\beta) ) \\
=&   - q^{-l}  \sum_{\substack{u \in \F^l \\ u^t \alpha = v^t}}  \sum_{ \substack{ \wp \in \PP(\F^n) \\ v^t (\wp)\neq 0 \\   \wp \nsubseteq \im(\beta) } } \alpha  \left( \begin{matrix} v^t \\  \varpi_{\wp}   \end{matrix} \right)^{-1}  \left( \begin{matrix} 1 & 0 \\ 0 &  \varpi_{\wp} \beta \end{matrix} \right) \\
=&  - q^{-l}  \cdot q^{l-n} \sum_{ \substack{ \wp \in \PP(\F^n) \\ v^t (\wp)\neq 0 \\  \wp \nsubseteq \im(\beta) } } \alpha  \left( \begin{matrix} v^t \\  \varpi_{\wp}   \end{matrix} \right)^{-1}  \left( \begin{matrix} 1 & 0 \\ 0 &  \varpi_{\wp} \beta \end{matrix} \right)  \\
=& \alpha \pi_n (  \Psi_{n,v,0}(\beta) ) .
\end{align*}

Now suppose that $\beta \in \C(m, n-1)$, $v\in \F^n$, $\wp \in \PP(\F^n)$, and $v^t(\wp)\neq 0$. One has:

\begin{align*}
&  \pi_l (  \alpha  \Psi_{n,v,\wp}(\beta) ) \\
=& \sum_{\substack{u \in \F^l \\ u^t \alpha = v^t}} \sum_{ \substack{\ell \in \PP(\F^l)\\ \ell = \alpha(\wp) } }  \pi_l ( \Psi_{l,u,\ell}( \alpha_{\wp} \beta ) ) \\
=& q^{-l} \sum_{\substack{u \in \F^l \\ u^t \alpha = v^t}} \sum_{ \substack{\ell \in \PP(\F^l)\\ \ell = \alpha(\wp) } }   \left( \begin{matrix} u^t \\  \varpi_\ell  \end{matrix} \right)^{-1} \left( \begin{matrix} 1 & 0 \\ 0 & \alpha_{\wp} \beta \end{matrix} \right) \\
=& q^{-l} \sum_{\substack{u \in \F^l \\ u^t \alpha = v^t}} \sum_{ \substack{\ell \in \PP(\F^l)\\ \ell = \alpha(\wp) } }   \alpha  \left( \begin{matrix} v^t \\  \varpi_{\wp}   \end{matrix} \right)^{-1}  \left( \begin{matrix} 1 & 0 \\ 0 &  \beta \end{matrix} \right) \quad \mbox{ using (\ref{key result equation}) }  \\
=& q^{-l}  \cdot q^{l-n}  \cdot \alpha  \left( \begin{matrix} v^t \\  \varpi_{\wp}   \end{matrix} \right)^{-1}  \left( \begin{matrix} 1 & 0 \\ 0 &  \beta \end{matrix} \right) \\
=& \alpha \pi_n (\Psi_{n,v, \wp}(\beta)) .
\end{align*}

This completes the verification that $\pi$ is a $\C$-module homomorphism.
\end{proof}

\section{Applications of coinduction functor} \label{applications of coinduction}

Throughout this section, we assume that $k$ is a field of characteristic 0.

\subsection{Proofs of Theorem \ref{main-result-1} and Corollary \ref{infinite global dimension}}
Suppose that $\C$ is $\FI_G$ or $\VI$. Recall that $\C$ is locally Noetherian by \cite{GL}.

\begin{proof}[Proof of Theorem \ref{main-result-1}]
It suffices to prove that $k\C e_n$ is injective for each $n\in \Ob(\C)$. We prove this by induction on $n$.

By Corollary \ref{kCe_0 is injective}, $k\C e_n$ is injective when $n=0$. Now suppose that $k\C e_n$ is injective when $n=m$ for some $m\in \Ob(\C)$. By Lemma \ref{eckmann-shapiro2}, $\q(k\C e_m)$ is injective. But $k\C e_{m+1}$ is a direct summand of $\q(k\C e_m)$ by Theorems \ref{main result on q 1} and \ref{main result on q 2}. It follows that $k\C e_n$ is injective for $n=m+1$.
\end{proof}

To carry out the inductive argument in the above proof, we need to know that $\q(k\C e_m)$ contains a direct summand isomorphic to $k\C e_{m+1}$. This was verified for the categories $\FI_G$ and $\VI$ by the explicit computations in Sections \ref{FI_G section} and \ref{VI section}; we do not know if a conceptual proof can be given based on certain underlying combinatorial properties of the category $\C$.

\begin{proof}[Proof of Corollary \ref{infinite global dimension}]
Suppose $V$ is a finitely generated $\C$-module and
\begin{equation*}
0 \to P^{-r} \to \cdots \to P^{-1} \to P^0 \to V \to 0
\end{equation*}
is an exact sequence where $P^0, \ldots, P^{-r}$ are finitely generated projective $\C$-modules. Then there are short exact sequences $0 \to Q^{i-1} \to P^i \to Q^i \to 0$ for $i=0, \ldots, -(r-1)$ where $Q^0, \ldots, Q^{-r}$ are finitely generated $\C$-modules such that $Q^0=V$ and $Q^{-r}=P^{-r}$. It follows from Theorem \ref{main-result-1} that these short exact sequences split; in particular, $V$ is a direct summand of $P^0$.
\end{proof}

\subsection{Torsion-free modules}
Suppose that $\C$ is $\FI_G$.

\begin{definition}
A $\C$-module $F$ is \emph{torsion-free} if $\Hom_\C(T, F)=0$ for all finite dimensional $\C$-modules $T$.
\end{definition}

\begin{lemma} \label{torsion-pair}
Suppose that $\C$ is $\FI_G$. Let $V$ be a finitely generated $\C$-module. Then there exists a short exact sequence
\begin{equation*}
0 \longrightarrow T \longrightarrow V \longrightarrow F \longrightarrow 0
\end{equation*}
such that $T$ is a finite dimensional $\C$-module and $F$ is a torsion-free $\C$-module.
\end{lemma}
\begin{proof}
Since $V$ is Noetherian, there exists a maximal finite dimensional submodule $T$ of $V$. It is plain that $V/T$ is torsion-free.
\end{proof}

\begin{lemma} \label{nonzero hom to projective}
Suppose that $\C$ is $\FI_G$. Let $F$ be a finitely generated torsion-free $\C$-module. If $F\neq 0$, then there exists $n\in \Ob(\C)$ such that $\Hom_\C(F, k\C e_n)\neq 0$.
\end{lemma}
\begin{proof}
Since $F\neq 0$, there exists a smallest $a\in \Ob(\C)$ such that $F(a)\neq 0$. Thus, $\s^a(F)(0) \neq 0$.
Choose a nonzero element $s$ of $\s^a(F)(0)$ and let
\begin{equation*}
f: k\C e_0 \longrightarrow \s^a(F), \quad \alpha \mapsto \alpha s.
\end{equation*}
Since $F$ is torsion-free, the homomorphism $f$ is injective. But $k\C e_0$ is an injective $\C$-module by Theorem \ref{main-result-1}. Thus, there exists a nonzero homomorphism from $\s^a(F)$ to $k\C e_0$. It follows from Lemma \ref{eckmann-shapiro2} that
\begin{equation*}
\Hom_\C(F, \q^a(k\C e_0)) = \Hom_\C(\s^a(F), k\C e_0) \neq 0.
\end{equation*}
But by Theorem \ref{main result on q 1}, $\q^a(k\C e_0)$ is isomorphic to $k\C e_{n_1} \oplus \cdots \oplus k\C e_{n_r}$ for some $n_1,\ldots, n_r\in \Ob(\C)$. Hence, the result follows.
\end{proof}

Suppose $V$ is a finitely generated $\C$-module. Then there exists $l \in \Ob(\C)$ such that $V$ is generated in degrees $\leqslant l$; clearly, one has $\Hom_\C ( V, k\C e_n) = 0$ for all $n>l$.

\begin{notation}
For any finitely generated $\C$-module $V$, let
\begin{equation*}
\kappa(V) = \sum_{n \in \Ob(\C)} \kappa (V, n),
\end{equation*}
where $\kappa(V, n) = \dim_k \Hom_{\C} (V, k\C e_n)$ for each $n\in \Ob(\C)$.
\end{notation}

\begin{proposition} \label{injection to projective}
Suppose that $\C$ is $\FI_G$. Let $F$ be a finitely generated torsion-free $\C$-module. Then there exists an injective homomorphism from $F$ to $k\C e_{n_1} \oplus \cdots \oplus k\C e_{n_r}$ for some $n_1,\ldots, n_r\in \Ob(\C)$.
\end{proposition}
\begin{proof}
If $\kappa(F)=0$, then $F=0$ by Lemma \ref{nonzero hom to projective}. We shall prove the proposition  by induction on $\kappa(F)$. Suppose $\kappa(F)>0$. Then there exists a nonzero homomorphism $f: F \to k\C e_m$ for some $m\in \Ob(\C)$. Let  $E=\Ker(f)$ and $W=\im(f)$. We have a short exact sequence
\begin{equation} \label{ses for induction}
0 \longrightarrow E \longrightarrow F \longrightarrow W \longrightarrow 0.
\end{equation}
By Theorem \ref{main-result-1}, the $\C$-module $k\C e_n$ is injective, so
\begin{equation*}
\kappa(E,n) + \kappa(W,n) = \kappa (F, n) \quad \mbox{ for each } n\in \Ob(\C).
\end{equation*}
Therefore, $\kappa(E)+ \kappa(W) = \kappa(F)$. But $\kappa(W,m)>0$, so $\kappa(E)<\kappa(F)$.
Observe that $E$ is a finitely generated torsion-free $\C$-module. By induction hypothesis, there exists an injective homomorphism from $E$ to $k\C e_{n_1} \oplus \cdots \oplus k\C e_{n_r}$ for some $n_1,\ldots, n_r\in \Ob(\C)$. Since we also have $W \hookrightarrow k\C e_m$ and the $\C$-modules $k\C e_1, \ldots, k\C e_r$, $k\C e_m$ are injective, it follows from (\ref{ses for induction}) by a standard argument that there exists an injective homomorphism from $F$ to $k\C e_{n_1} \oplus \cdots \oplus k\C e_{n_r}\oplus k\C e_m$.
\end{proof}

\subsection{Proof of Theorem \ref{main-result-2}}
Suppose that $\C$ is $\FI_G$.

\begin{proof}[Proof of Theorem \ref{main-result-2}]
(i) Suppose that $V$ is a finitely generated injective $\C$-module. By Lemma \ref{torsion-pair}, there is a short exact sequence $0 \to T \to V \to F \to 0$ where $T$ is a finite dimensional $\C$-module and $F$ is a torsion-free $\C$-module. For any $\C$-module $U$, there is a long exact sequence
\begin{equation*}
\cdots \to \Hom_\C (U, F) \to \Ext^1_\C (U,T) \to \Ext^1_\C (U, V) \to \cdots .
\end{equation*}
Since $V$ is injective, one has $\Ext^1_\C(U,V)=0$. Since $F$ is torsion-free, one has $\Ext^1_\C(U, T)=0$ whenever $U$ is finite dimensional.

Choose $l\in \Ob(\C)$ such that $T(m)=0$ for all $m>l$. Suppose $W$ is a finitely generated $\C$-module. By Lemma \ref{reduction to C_n}, we can choose $n>l$ such that: $\Ext^1_{\C}(W, T) = \Ext^1_{\C_n} (\jmath^*(W), \jmath^*(T))$ where $\jmath: \C_n\hookrightarrow \C$ denotes the inclusion functor.
Let $U= \jmath_*(\jmath^*(W))$. Observe that $\jmath^*(U) = \jmath^*(W)$ and $\jmath_*(\jmath^*(T))=T$. Thus, by Lemma \ref{eckmann-shapiro1}, one has
\begin{equation*}
\Ext^1_{\C_n} (\jmath^*(W), \jmath^*(T)) = \Ext^1_{\C} (U, T) = 0.
\end{equation*}
Therefore, $\Ext^1_{\C}(W, T) = 0$. It follows from Corollary \ref{injectivity in C-mod} that $T$ is an injective $\C$-module.

We deduce that $V$ is isomorphic to $T\oplus F$, so $F$ is an injective $\C$-module. By Proposition \ref{injection to projective}, it follows that $F$ is a direct summand of a projective $\C$-module. Therefore, $F$ is a projective $\C$-module.

(ii) Suppose that $V$ is a finitely generated $\C$-module generated in degrees $\leqslant l$. We shall prove the result by induction on $l$.

We may assume, without loss of generality, that $V$ has no projective direct summands. By Lemma \ref{torsion-pair}, there is a short exact sequence $0 \to T \to V \to F \to 0$ where $T$ is a finite dimensional $\C$-module and $F$ is a torsion-free $\C$-module. By Proposition \ref{injection to projective}, there is an injective homomorphism
\begin{equation*}
f: F \to k\C e_{n_1} \oplus \cdots \oplus k\C e_{n_r} \quad
\mbox{ for some } n_1,\ldots, n_r\in \Ob(\C).
\end{equation*}

Observe that $F$ is generated in degrees $\leqslant l$ and has no projective direct summands. Since there is no nonzero homomorphism from $F$ to $k\C e_n$ for any $n>l$, we can assume $n_1, \ldots, n_r \leqslant l$. We claim that there is also no nonzero homomorphism from $F$ to $k\C e_l$. Indeed, the image of any homomorphism from $F$ to $k\C e_l$ is a submodule of $k\C e_l$ generated in degrees $\leqslant l$, but any such submodule of $k\C e_l$ is a direct summand of $k\C e_l$ and hence a projective $\C$-module. Since $F$ has no projective direct summands, it follows that any homomorphism from $F$ to $k\C e_l$ must be 0. If $l=0$, this implies that $F=0$, so $V=T$, and we are done by Lemma \ref{inj res of finite dim module}.

Suppose $l>0$. By the above observations, we can assume that  $n_1, \ldots, n_r \leqslant l-1$. Let $W$ be the cokernel of $f$. Since $W$ is generated in degrees $\leqslant l-1$, it follows by induction hypothesis that $W$ has a finite injective resolution
\begin{equation*}
0 \to W  \to J^1 \to \cdots \to J^a \to 0
\end{equation*}
 in the category $\C\module$. From this, we obtain a finite injective resolution
 \begin{equation*}
0 \to F \stackrel{f}{\to} J^0 \to J^1 \to \cdots \to J^a \to 0
\end{equation*}
where $J^0 = k\C e_{n_1} \oplus \cdots \oplus k\C e_{n_r}$. Recall that by Lemma \ref{inj res of finite dim module}, the $\C$-module $T$ has a finite injective resolution in $\C\module$. We conclude by the horseshoe lemma (see \cite[page 37]{Weibel}) that $V$ has a finite injective resolution in $\C\module$.
\end{proof}

\section{Homological approach to representation stability}  \label{last section}

Throughout this section, we assume that $k$ is a splitting field for $G$ of characteristic 0.

\subsection{Simple modules of wreath product groups}
Let $\C$ be $\FI_G$. Then $G_n = G \wr S_n$.

For each $m, l \in \Ob(\C)$, we consider $G_m \times G_l$ as a subgroup of $G_{m+l}$ via
\begin{equation*}
 G_m \times G_l \hookrightarrow G_{m+l}, \quad \big( (f_1,c_1), (f_2,c_2) \big) \mapsto (f_1,c_1)\odot (f_2,c_2).
\end{equation*}
If $X$ is a $kG_m$-module, and $Y$ is a $kG_l$-module, we set
\begin{equation*}
X \circledast Y = kG_{m+l} \otimes_{k(G_m\times G_l)} (X\otimes_k Y).
\end{equation*}
If $A$ is a $k G$-module, and $E$ is a $k S_n$-module, we write $A\wr E$ for the $k G\wr S_n$-module $A^{\otimes n} \otimes_k E$.

We denote by $\irr(G)=\{\chi_1, \ldots, \chi_r\}$ the set of isomorphism classes of simple $k G$-modules; in particular, let $\chi_1 \in \irr(G)$ be the trivial class. For each $\chi\in\irr(G)$, let $A(\chi)$ be a simple $k G$-module belonging to the isomorphism class $\chi \in \irr(G)$. Recall that the isomorphism classes of simple $k S_n$-modules are parametrized by the partitions of $n$. For each partition $\lambda$ of $n$, let $E(\lambda)$ be a simple $k S_n$-module whose isomorphism class corresponds to $\lambda$.

If $\la$ is a partition-valued function on $\irr(G)$, we let
\begin{equation*}
\widetilde{L}(\la) = \Big( A(\chi_1) \wr E(\la(\chi_1)) \Big) \circledast \cdots \circledast \Big( A(\chi_r) \wr E(\la(\chi_r)) \Big).
\end{equation*}
The following result on the classification of simple $k G\wr S_n$-modules is well-known (see \cite{Macdonald}).

\begin{theorem}
Suppose that $k$ is a splitting field for $G$ of characteristic 0. Then the set of $\widetilde{L}(\la)$ for all partition-valued functions $\la$ on $\irr(G)$ with $|\la|=n$ is a complete set of non-isomorphic simple $k G \wr S_n$-modules.
\end{theorem}

\subsection{Proof of Theorem \ref{rep stable thm}}
The homological properties of the category $\C\module$ proved in the previous section allow one to give a quick proof that finite generation implies condition (RS3).

For any partition-valued function $\la$ on $\irr(G)$, and integer $n\geqslant |\la|+a$ where $a$ is the biggest part of $\la(\chi_1)$, let
\begin{equation*}
L(\la)_n = \widetilde{L}(\la[n]).
\end{equation*}

\begin{proof}[Proof of Theorem \ref{rep stable thm}]
As explained in subsection \ref{rep stability}, it only remains to verify that condition (RS3) holds for finitely generated projective $\C$-modules. But by \cite[Theorem 11.18]{Dieck}, any finitely generated projective $\C$-module is a direct sum of $\C$-modules of the form $k \C e_m \otimes_{k G_m} L(\la)_m$ for some $m\in \Ob(\C)$ and partition-valued function $\la$ on $\irr(G)$.

Let $n\geqslant 2m$. Denote by $\mu$ the trivial partition of $n-m$. One has
\begin{equation*}
k\C (m,n) = k G_m \circledast \Big(A(\chi_1) \wr E(\mu) \Big).
\end{equation*}
Therefore,
\begin{multline*}
 k\C (m,n) \otimes_{k G_m} L(\la)_m
=  L(\la)_m \circledast \Big(A(\chi_1) \wr E(\mu) \Big)\\
= \Big( A(\chi_1) \wr E(\la[m](\chi_1)) \Big) \circledast \Big(A(\chi_1) \wr E(\mu) \Big)  \circledast \left( \bigCircledast_{i=2}^r  A(\chi_i) \wr E(\la(\chi_i)) \right).
\end{multline*}
The rest of the proof is same as \cite[Lemma 2.3]{Hemmer}. By Pieri's formula,
\begin{equation*}
 \Big( A(\chi_1) \wr E(\la[m](\chi_1)) \Big) \circledast \Big(A(\chi_1) \wr E(\mu) \Big) \\
= \bigoplus_{\nu \in P(n)} A(\chi_1) \wr E(\nu),
\end{equation*}
where $P(n)$ denotes the set of all partitions $\nu$ whose Young diagram can be obtained from the Young diagram of $\la[m](\chi_1)$ by adding $n-m$ boxes with no two in the same column. Let $\hbar: P(n) \to P(n+1)$ be the map which assigns to $\nu\in P(n)$ the partition $\hbar(\nu)\in P(n+1)$ whose Young diagram is obtained from the Young diagram of $\nu$ by adding a box in the first row. It is plain that $\hbar$ is injective. Since $n\geqslant 2m$, the map $\hbar$ must also be surjective. The result follows.
\end{proof}

\end{document}